\documentclass[11pt,a4paper]{article}
\usepackage{}
\usepackage{amsfonts}
\usepackage{amssymb}
\usepackage{bm}
\usepackage{colortbl}
\usepackage{subfig}
\usepackage{makecell}
\usepackage{epsf,epsfig,amsfonts,amsgen,amsmath,amstext,amsbsy,amsopn,amsthm,lineno}
\usepackage{color}
\usepackage{graphicx,tikz}
\usepackage{cite}

\allowdisplaybreaks
\makeatletter

\newcommand{\Rmnum}[1]{\expandafter\@slowromancap\romannumeral #1@}
\makeatother

\newtheorem{theorem}{Theorem}

\newtheorem{lemma}{Lemma}
\newtheorem{remark}{Remark}

\theoremstyle{definition}
\newtheorem{definition}{Definition}

\newtheorem{proposition}{Proposition}

\tikzstyle{vertex}=[circle, draw, inner sep=0pt, minimum size=3pt]

\numberwithin{equation}{section} %%公式编号跟随章节
\allowdisplaybreaks  %%公式分页展示

\def\qed{\hfill$\Box$\vspace{12pt}}
\long\def\delete#1{}

\begin{document}
\title
{\bf\Large  Extremality of graph entropy based on degrees of uniform hypergraphs with few edges\thanks{Supported by NSFC (Nos.~11531011, 11571135, 11671320 and 11701441) and the China Postdoctoral Science Foundation (No. 2016M600813).}}

\date{}
\author{
\small Dan Hu$^a$,
Xueliang Li$^b$, Xiaogang Liu$^a$,
Shenggui Zhang$^{a,}$\thanks{Corresponding author. E-mail addresses: hudan@mail.nwpu.edu.cn, lxl@nankai.edu.cn, xiaogliu@nwpu.edu.cn, sgzhang@nwpu.edu.cn} \\[2mm]
\small $^a$Department of Applied Mathematics, Northwestern Polytechnical University,\\
\small Xi'an, Shaanxi 710072, P.R.~China\\
\small $^b$Center for Combinatorics, Nankai University,\\
\small Tianjin 300071, P.R.~China}
\maketitle
\begin{abstract}
%Many graph invariants have been used for the construction of entropy-based measures to characterize the structure of graphs.

Let $\mathcal{H}$ be a hypergraph with $n$ vertices. Suppose that $d_1,d_2,\ldots,d_n$ are degrees of the vertices of $\mathcal{H}$. The \emph{$t$-th graph entropy based on degrees of $\mathcal{H}$} is defined as
$$
I_d^t(\mathcal{H}) =-\sum_{i=1}^{n}\left(\frac{d_i^{t}}{\sum_{j=1}^{n}d_j^{t}}\log\frac{d_i^{t}}{\sum_{j=1}^{n}d_j^{t}}\right) =\log\left(\sum_{i=1}^{n}d_i^{t}\right)-\sum_{i=1}^{n}\left(\frac{d_i^{t}}{\sum_{j=1}^{n}d_j^{t}}\log d_i^{t}\right),
$$
where $t$ is a real number and the logarithm is taken to the base two. In this paper we obtain upper and lower bounds of $I_d^t(\mathcal{H})$ for $t=1$, when $\mathcal{H}$ is among all uniform supertrees, unicyclic uniform hypergraphs and bicyclic uniform hypergraphs, respectively.
%In this paper, we will discuss the extremal properties of the above graph entropy based on degree powers. we consider the graph entropies based on the given information functional by using degree powers. Further, extremal properties of graph entropies of certain hypergraphs haven been studied.

\smallskip

\end{abstract}

\section{Introduction}

Throughout this paper, logarithms are always taken to the base two. Let $p =(p_1, p_2, \ldots , p_n)$ be a probability distribution, that is, $0 \leq p_i \leq1 $ and $\sum_{i=1}^{n}p_i=1$. The \emph{Shannon's entropy} is defined as
\begin{equation}\label{ShanEntr1}
 I(p)=-\sum_{i=1}^{n}(p_i\log p_i).
\end{equation}
This notion was first proposed by Shannon \emph{et al.} \cite{Shannon} in 1949, which is one of the most important metrics in information theory as a measure of unpredictability of information content.

Applying Shannon's entropy formula (\ref{ShanEntr1}) with a probability distribution defined on the vertex set or edge set of a graph, one can obtain a numerical value, which is usually called a \emph{graph entropy}. Of course, different probability distributions defined on graphs contribute to different graph entropies.  Up until now, a lots of graph entropies have been proposed (see \cite{Dehmer, Dehmer2, Mowshowitz1, Rashevsky, Trucco,BonchevI, Mowshowitz2, Mowshowitz3, Mowshowitz4, Korner, Korner5, xLi, Gabor-Simonyi}). For more results on the theory and applications of the graph entropy, we refer the reader to four survey papers \cite{Dehmer2, xLi1, Gabor-Simonyi95, Gabor-Simonyi}.

Let $G$ be a finite, undirected and connected graph with vertex set $V_G=\{v_1,v_2,\ldots,v_n\}$. Define a mapping
$f:V_G\rightarrow \mathbb{R}^{+}$, and a probability distribution $p=(p(v_1), p(v_2),\ldots,p(v_n))$, where
\[p(v_i)=\frac{f(v_i)}{\sum_{j=1}^{n}f(v_j)}\]
for each vertex $v_i\in V_G$. The graph entropy of $G$ based on $f$ is defined as
\begin{equation}\label{GraEnto1}
I_f(G)=-\sum_{i=1}^{n}\left(\frac{f(v_i)}{\sum_{j=1}^{n}f(v_j)}\log\frac{f(v_i)}{\sum_{j=1}^{n}f(v_j)}\right).
\end{equation}
Such a graph entropy was introduced by Dehmer in \cite{Dehmer}. Note that in Equation (\ref{GraEnto1}), $f$ can be any mapping. Thus, by (\ref{GraEnto1}), Cao \emph{et al.} \cite{Cao} gave a novel graph entropy based on the degrees of graphs as follows.

\begin{definition}[See \cite{Cao}]\label{Def:2}
Let $G$ be a connected graph with vertex set $V_G=\{v_1,v_2,\ldots,v_n\}$.
Denote by $d_i$ the degree of the vertex $v_i$.
Then the \emph{$t$-th graph entropy based on degrees of $G$} is defined as
\[I_d^t(G) =-\sum_{i=1}^{n}\left(\frac{d_i^{t}}{\sum_{j=1}^{n}d_j^{t}}\log\frac{d_i^{t}}{\sum_{j=1}^{n}d_j^{t}}\right) =\log\left(\sum_{i=1}^{n}d_i^{t}\right)-\sum_{i=1}^{n}\left(\frac{d_i^{t}}{\sum_{j=1}^{n}d_j^{t}}\log d_i^{t}\right),\]
where $t$ is a real number.
\end{definition}

In \cite{Cao}, Cao \emph{et al.} studied the extremal properties of $I_d^t(G)$ for $t=1$, when $G$ is in the class of trees, unicyclic graphs, bicyclic graphs, chemical trees or chemical graphs. For $t>0$, Cao \emph{et al.} conjectured that the path on $n$ vertices is the unique tree on $n$ vertices that maximizes $I_d^t(G)$ in the class of trees, and the star on $n$ vertices is the unique tree on $n$ vertices that minimizes $I_d^t(G)$ in the class of trees.
In \cite{Ilic}, Ili\'{c} showed that the upper bound is correct for $t>0$, and the lower bound is correct only for $t\ge1$ in the conjecture.

In this paper, we extend the results of \cite{Cao} to $k$-uniform hypergraphs.
 A \emph{hypergraph} $\mathcal{H} = (V_\mathcal{H},E_\mathcal{H})$ with $n$ vertices and $m$ edges consists of a set of vertices $V_\mathcal{H} = \{1, 2, \ldots, n\}$ and a set of edges $E_\mathcal{H}= \{e_1, e_2, \ldots, e_m\}$, where
$e_i$ is a nonempty subset of $V_\mathcal{H}$ for $i = 1, 2, \ldots, m$. If $|e_i| = k$ for $i = 1, 2, \ldots,m$, then $\mathcal{H}$ is called a \emph{$k$-uniform
hypergraph}. Clearly, ordinary graphs are referred to as $2$-uniform
hypergraphs. A $k$-uniform hypergraph $\mathcal{H}$ is called \emph{simple} if there are no multiple edges in $\mathcal{H}$, that is, all edges in $\mathcal{H}$ are distinct. In  the sequel of this paper, we assume that all hypergraphs considered are simple and $k$-uniform with $k \geq 3$.

In a hypergraph, a vertex $v$ is said to be \emph{incident} to an edge $e$ if $v \in e$. Two vertices are said to be \emph{adjacent} if
there is an edge that contains both of these vertices. Two edges are said to be \emph{adjacent} if their intersection is not empty.
For a $k$-uniform hypergraph $\mathcal{H}$,  the \emph{degree} $d_v$ of a vertex $v \in V_\mathcal{H}$ is defined as $d_v =
|\{e_j : v \in e_j \in E_\mathcal{H}\}|$. A vertex of degree one is called a \emph{pendent vertex}. Otherwise, it is called a \emph{non-pendent vertex}. An edge $e \in E_\mathcal{H}$ is called a \emph{pendent edge} if $e$ contains exactly $k-1$ pendent vertices. Otherwise, it is called a \emph{non-pendent edge}.

%In a hypergraph $\mathcal{H}$, a path of length $q$ is defined to be an alternating sequence
%of vertices and edges $v_1, e_1, v_2, e_2,\ldots, v_q , e_q, v_{q+1}$ such that
%
%(1) $v_1,\ldots, v_{q+1}$ are all distinct vertices of $\mathcal{H}$,
%
%(2) $e_1,\ldots, e_q$ are all distinct edges of $\mathcal{H}$,
%
%(3) $v_r, v_{r+1}\in e_r$ for $r = 1,\ldots, q$.
%
%If $q> 1$ and $v_1 = v_{q+1}$, then this path is called a cycle of length $q$. A hypergraph $\mathcal{H}$
%is connected if there exists a path starting at $v$ and terminating at $u$ for all $v$, $u\in V$,
%and is called acyclic if it contains no cycle.
A \emph{walk} $W$ of length $l$ in $\mathcal{H}$ is a sequence of alternating vertices and edges: $v_0e_1v_1e_2\cdots e_lv_l$, where $\{v_{i-1}, v_i\}\subseteq e_i$ for $i=1,\ldots,l$. If $v_0=v_l$, then $W$ is called a \emph{circuit}. A walk of $\mathcal{H}$ is called a \emph{path} if no vertices and edges are repeated. A circuit $\mathcal{H}$ is called a \emph{cycle} if no vertices and edges are repeated except $v_0=v_l$. The hypergraph $\mathcal{H}$ is said to be \emph{connected} if every two vertices are connected by a walk. A hypergraph $\mathcal{H}$ is called a \emph{linear hypergraph} if each
pair of the edges of $\mathcal{H}$ has at most one common vertex. Otherwise, it is called a \emph{non-linear hypergraph}.

The following concept of \emph{power hypergraph} was introduced in \cite{Hu}.

\begin{definition}[See \cite{Hu}]\label{Def:3}
Let $G=(V,E)$ be an ordinary graph. For an integer $k\geq3$, the \emph{$k$-th power} of
$G$, denoted by $G^{k}:=(V^{k},E^{k})$, is defined as the $k$-uniform hypergraph with the set of vertices
$V^{k}=V\cup(\cup_{e\in E}\{i_{e,1},\ldots,i_{e,k-2}\})$ and the set of edges $E^{k}=\{e \cup \{i_{e,1},\ldots,i_{e,k-2}\}\big|e\in E\}$, where $i_{e,1},\ldots,i_{e,k-2}$ are new added vertices for $e$.
\end{definition}
\begin{definition}[See \cite{Hu, Li}]\label{Def:4}
The $k$-th power of an ordinary tree is called a $k$-uniform \emph{hypertree}.
\end{definition}

The \emph{star} with $n$ vertices, denoted by $S_n:= (\{v_1, \ldots, v_n\}, E_S)$, is a graph
with the vertex set $\{v_1, \ldots, v_n\}$ in which $v_iv_j\in E_S$ if and only if $i = 1$ or $j = 1$.

\begin{definition}[See \cite{Hu}]\label{Def:Hstar}
The $k$-th power of $S_n$, denoted by $S^{k}_n$, is called a \emph{hyperstar} (see Figure \ref{EHypers} (a) for an example).
\end{definition}

\begin{definition}[See \cite{Fan}]\label{Def:6}
Let $\mathcal{H}$ be a $k$-uniform hypergraph with $n$ vertices, $m$ edges and $l$ connected components. The \emph{cyclomatic number} of $\mathcal{H}$ is defined to be
$$
c(\mathcal{H})=m(k-1)-n+l.
$$
 So, a hypergraph $\mathcal{H}$ can be called a \emph{$c(\mathcal{H})$-cyclic hypergraph}.
\end{definition}

 \begin{definition}[See \cite{Li}]\label{Def:acyclic}
A hypergraph is said to be \emph{acyclic} if it does not contain any cycles.
 \end{definition}

A connected hypergraph $\mathcal{H}$ is acyclic if and only if $c(\mathcal{H})=0$.

\begin{proposition}[See \cite{Li}]\label{Prop:1}
A connected $k$-uniform hypergraph with $n$ vertices and $m$ edges is \emph{acyclic} if and only if $m(k-1)=n-1$.
\end{proposition}

\begin{definition}[See \cite{Li}]\label{Def:5}
A \emph{supertree} is a hypergraph which is both connected and acyclic.
\end{definition}

From Definitions \ref{Def:4} and \ref{Def:5}, we know that all $k$-uniform hypertrees are supertrees. Conversely, a $k$-uniform supertree $\mathcal{T}$ with at least two edges is a $k$-uniform hypertree if and only if each edge of $\mathcal{T}$ contains at most two non-pendent vertices. For example, Figure \ref{supertree} depicts a $5$-uniform supertree, which is not a hypertree.

\begin{figure}[h]
\vspace{0.5cm}
\centering
\begin{tikzpicture}[scale=2.25]
\filldraw [black]
(0.654,0.475)circle (0.7pt)
(0.2000 ,0.475)circle (0.7pt)
  ( 1.0580,   0.475)circle (0.7pt)
    (0.4520,    0.475)circle (0.7pt)
    (0.8560,    0.475)circle (0.7pt)
    (1.2560,    0.475)circle (0.7pt)
    %(0.9560,    0.475)circle (0.3pt)
    (0.01560,    0.475)circle (0.7pt)
    %(-0.070,    0.475)circle (0.3pt)
   (-0.1540,    0.475)circle (0.7pt)
   %(0.10480 ,   0.475)circle (0.7pt)
   (-0.3560 ,   0.475)circle (0.7pt)
   (-0.60,    0.475)circle (0.7pt)
    (-0.770,    0.475)circle (0.7pt)
   (-0.970,    0.475)circle (0.7pt)
   %(0.10480 ,   0.475)circle (0.7pt)
   (-1.170 ,   0.475)circle (0.7pt)
   (0.2000 ,0.675)circle (0.7pt)
   (0.2000 ,0.875)circle (0.7pt)
   (0.2000 ,1.075)circle (0.7pt)
   (0.2000 ,1.275)circle (0.7pt);
\draw[rotate around={0:(0.9,0.475)},blue] (0.9,0.475) ellipse (0.55 and 0.12);
\draw[rotate around={0:(-0.77,0.475)},blue] (-0.77,0.475) ellipse (0.55 and 0.12);
\draw[rotate around={90:(0.20,0.9)},blue] (0.20,0.9) ellipse (0.55 and 0.12);
\draw[rotate around={0:(0.00,0.475)},blue] (0.0,0.475) ellipse (0.55 and 0.12);
\end{tikzpicture}
\caption{A $5$-uniform supertree}\label{supertree}
\end{figure}
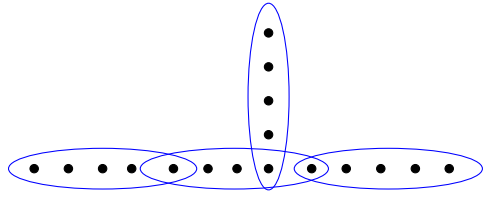

%\hd{We know that if each edge of a $k$-uniform supertree $\mathcal{T}$ contains at most two non-pendent vertices, then we obtain an ordinary tree $T$ by iteratively applying the process that deleting $k-2$ pendent vertices in each edge $e$ in $\mathcal{T}$. Therefore, $\mathcal{T}$ is the $k$-th power of $T$. From Definitions \ref{Def:4}, we know that $\mathcal{T}$ a $k$-uniform hypertree. Conversely, if a $k$-uniform supertree $\mathcal{T}$ with at least two edges is a $k$-uniform hypertree, then each edge of $\mathcal{T}$ contains at most two non-pendent vertices.}

\begin{definition}[See \cite{Fan}]\label{Def:Hunic}
If $\mathcal{H}$ is connected and contains exactly one cycle,
then $\mathcal{H}$ is called a \emph{unicyclic hypergraph}  (see Figure \ref{EHypers} (b) for an example).
\end{definition}

 \begin{proposition}[See \cite{Berge1}]\label{Prop:0}
 If $\mathcal{H}$ is a connected hypergraph with $n$ vertices and $m$ edges, then it has a unique cycle if and only if $\sum_{i=1}^{m}(|e_i|-1)=n$.
\end{proposition}
A connected hypergraph $\mathcal{H}$ is unicyclic if and only if $c(\mathcal{H})=1$.
The following result follows from Definition \ref{Def:Hunic} and Proposition \ref{Prop:0} immediately.

\begin{proposition}
%[See \cite{Berge1}]
\label{Prop:2}
A connected $k$-uniform hypergraph $\mathcal{H}$ with $n$ vertices and $m$ edges is a unicyclic $k$-uniform hypergraph if and only if $m(k-1)=n$.
\end{proposition}

\begin{definition}[See \cite{Fan}]\label{Def:Hbicyc}
A connected hypergraph $\mathcal{H}$ is called \emph{bicyclic} if $c(\mathcal{H})=2$  (see Figure \ref{EHypers} (c) and (d) for examples).
 \end{definition}

The following result follows from Definitions \ref{Def:6} and \ref{Def:Hbicyc} immediately.

\begin{proposition}\label{Prop:3}
 A connected $k$-uniform hypergraph $\mathcal{H}$ with $n$ vertices and $m$ edges is a bicyclic $k$-uniform hypergraph if and only if $m(k-1)=n+1$.
\end{proposition}

%Figure \ref{EHypers} depicts examples of a hyperstar, nonlinear unicyclic and bicyclic hypergraphs.

 \begin{frame}{}
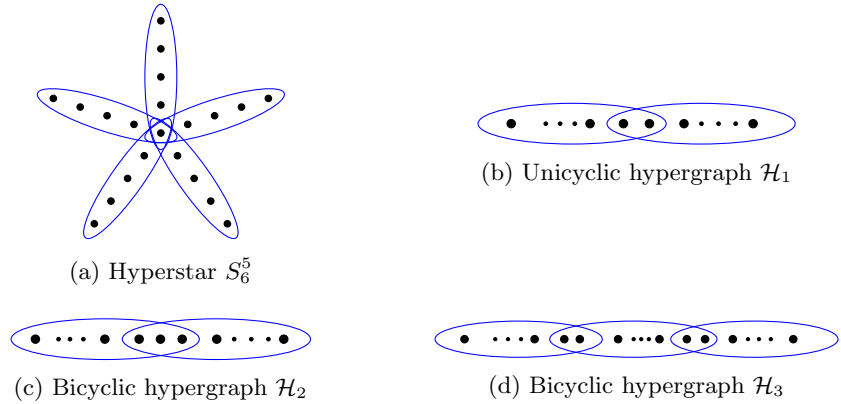
\begin{figure}[htbp]
\vspace{-0.45cm}
\centering
 \subfloat[Hyperstar $S^{5}_6$]{
\label{DirSample_frac01}
  \begin{minipage}[c]{0.4\textwidth}
\centering
\begin{tikzpicture}[scale=1.75]
\draw[rotate around={55:(0,0)},blue] (0,0) ellipse (0.55 and 0.12);
\filldraw [black] (0,0) circle (0.7pt)
(0.5,0)circle (0.7pt)
(0.654,0.475)circle (0.7pt)
(0.250,0.769)circle (0.7pt)
(-0.154,0.475)circle (0.7pt)
(0.25000 ,   0.34380)circle (0.7pt)
( -0.2500,   -0.3438)circle (0.7pt)
 (0.7500 ,  -0.3438)circle (0.7pt)
  ( 1.0580,    0.6062)circle (0.7pt)
    (0.2500 ,   1.1942)circle (0.7pt)
    (-0.5580,    0.6062)circle (0.7pt)
    %sec
    (0.1250,    0.1719)circle (0.7pt)
    (0.3750,    0.1719)circle (0.7pt)
    (0.4520,    0.4094)circle (0.7pt)
    (0.2500,    0.5564)circle (0.7pt)
    (0.0480,    0.4094)circle (0.7pt)
    %four
    (-0.1250,   -0.1719)circle (0.7pt)
    (0.6250,   -0.1719)circle (0.7pt)
    (0.8560,    0.5406)circle (0.7pt)
    (0.2500,    0.9816)circle (0.7pt)
   (-0.3560,    0.5406)circle (0.7pt);
\draw[rotate around={125:(0.5,0)},blue] (0.5,0) ellipse (0.55 and 0.12);
\draw[rotate around={18:(0.654,0.475)},blue] (0.654,0.475) ellipse (0.55 and 0.12);
\draw[rotate around={90:(0.250,0.769)},blue] (0.250,0.769) ellipse (0.55 and 0.12);
\draw[rotate around={-18:(-0.154,0.475)},blue] (-0.154,0.475) ellipse (0.55 and 0.12);
\end{tikzpicture}
\end{minipage}
}
\subfloat[Unicyclic hypergraph $\mathcal{H}_1$]{
\label{DirSample_frac02}
  \begin{minipage}[c]{0.40\textwidth}
\centering
\begin{tikzpicture}[scale=2.25]
\filldraw [black]
(0.654,0.475)circle (0.7pt)
(0.3000 ,0.475)circle (0.7pt)
  ( 1.0580,   0.475)circle (0.7pt)
    (0.4520,    0.475)circle (0.7pt)
    (0.8560,    0.475)circle (0.3pt)
    (0.7560,    0.475)circle (0.3pt)
    (0.9560,    0.475)circle (0.3pt)
    (0.01560,    0.475)circle (0.3pt)
    (-0.070,    0.475)circle (0.3pt)
   (-0.1540,    0.475)circle (0.3pt)
   (0.10480 ,   0.475)circle (0.7pt)
   (-0.3560 ,   0.475)circle (0.7pt);
\draw[rotate around={0:(0.754,0.475)},blue] (0.754,0.475) ellipse (0.55 and 0.12);
\draw[rotate around={0:(0.00,0.475)},blue] (0.00,0.475) ellipse (0.55 and 0.12);
\end{tikzpicture}
\end{minipage}
}

\subfloat[Bicyclic hypergraph $\mathcal{H}_2$]{
\label{DirSample_frac02}
  \begin{minipage}[c]{0.40\textwidth}
\centering
\begin{tikzpicture}[scale=2.25]
\filldraw [black]
(0.604,0.475)circle (0.7pt)
(0.275000 ,0.475)circle (0.7pt)
  ( 0.99580,   0.475)circle (0.7pt)
    (0.4020,    0.475)circle (0.7pt)
    (0.8060,    0.475)circle (0.3pt)
    (0.7060,    0.475)circle (0.3pt)
    (0.9060,    0.475)circle (0.3pt)
    (-0.050,    0.475)circle (0.7pt)
    (-0.32070,    0.475)circle (0.3pt)
   (-0.17540,    0.475)circle (0.3pt)
   (-0.2540,    0.475)circle (0.3pt)
   (0.150480 ,   0.475)circle (0.7pt)
   (-0.4560 ,   0.475)circle (0.7pt);
\draw[rotate around={0:(0.6020,0.475)},blue] (0.6020,0.475) ellipse (0.55 and 0.12);
\draw[rotate around={0:(-0.0480,0.475)},blue] (-0.0480,0.475) ellipse (0.55 and 0.12);
\end{tikzpicture}
\end{minipage}
}
\subfloat[Bicyclic hypergraph $\mathcal{H}_3$]{
\label{DirSample_frac02}
  \begin{minipage}[c]{0.40\textwidth}
\centering
\begin{tikzpicture}[scale=2.0]
\filldraw [black]
(0.654,0.475)circle (0.7pt)
(0.3000 ,0.475)circle (0.7pt)
  ( 1.10580,   0.475)circle (0.7pt)
  ( 1.22580,   0.475)circle (0.7pt)
  ( 1.4080,   0.475)circle (0.7pt)
  ( 1.5080,   0.475)circle (0.3pt)
  ( 1.580,   0.475)circle (0.3pt)
  ( 1.6580,   0.475)circle (0.3pt)
  ( 1.80580,   0.475)circle (0.7pt)
    (0.4020,    0.475)circle (0.7pt)
    (0.8560,    0.475)circle (0.3pt)
    (0.8060,    0.475)circle (0.3pt)
    (0.7560,    0.475)circle (0.3pt)
    (0.9260,    0.475)circle (0.7pt)
    (0.01560,    0.475)circle (0.3pt)
    (-0.070,    0.475)circle (0.3pt)
   (-0.1540,    0.475)circle (0.3pt)
   (0.10480 ,   0.475)circle (0.7pt)
   (-0.3560 ,   0.475)circle (0.7pt);
\draw[rotate around={0:(0.754,0.475)},blue] (0.754,0.475) ellipse (0.55 and 0.12);
\draw[rotate around={0:(0.00,0.475)},blue] (0.00,0.475) ellipse (0.55 and 0.12);
\draw[rotate around={0:( 1.550,0.475)},blue] ( 1.55,0.475) ellipse (0.55 and 0.12);
\end{tikzpicture}
\end{minipage}
}
\caption{Examples on hyperstars, non-linear unicyclic or bicyclic hypergraphs}\label{EHypers}
\end{figure}
\end{frame}

The paper is structured as follows. In Section 2, we give some definitions and basic results. In Sections 3-5, we obtain upper and lower bounds of $I_d^t(\mathcal{H})$ for $t=1$, when $\mathcal{H}$ is among all uniform supertrees, unicyclic uniform hypergraphs and bicyclic uniform hypergraphs, respectively.
\section{Preliminaries}
In this section, we give some definitions and basic results which will be used later.
Let $\mathcal{H} = (V_\mathcal{H}, E_\mathcal{H})$ be a $k$-uniform hypergraph with $n$ vertices and $m$ edges. Define the non-increasing \emph{degree sequence} of $\mathcal{H}$ by $\pi(\mathcal{H})=(d_1,d_2,\ldots,d_n)$, that is, $d_1 \geq d_2 \geq \cdots\geq d_n$. Note that
\[\sum_{i=1}^{n}d_i=km.\]
By Definition \ref{Def:2}, we have
\begin{equation}\label{Equ:3.1}
I_d^1(\mathcal{H})=
\log\left(\sum_{i=1}^{n}d_i\right)-\sum_{i=1}^{n}\left(\frac{d_i}{\sum_{j=1}^{n}d_j}\log d_i\right) =\log(km)-\frac{1}{km}\sum_{i=1}^{n}(d_i\log d_i).
\end{equation}
Thus, for a class of $k$-uniform hypergraphs with given number of edges, in order to determine the extremal values of $I_d^1(\mathcal{H})$, we just need to determine the extremal
values of $\sum\limits_{i=1}^{n}(d_i\log d_i)$.  Then we define
$$
h(\mathcal{H}):=\sum_{i=1}^{n}(d_i\log d_i).
$$
In the following, we first study some properties of $h(\mathcal{H})$.

%Let $\pi=(d_1,\ldots,d_n)$ be a non-increasing sequence of nonnegative integers,
% we write $\pi= (d_1^{m_1},\ldots, d_n^{m_n})$
%, where exponents denote multiplicity.
%To start we state some notation. Denote by $\pi(\mathcal{H})=(d_1,d_2,\ldots,d_n)$ the degree sequence of $\mathcal{H}$, and

Li, Shao and Qi \cite[Definition 14]{Li} introduced the operation of \emph{moving edges} on hypergraphs, which is stated as follows.

\begin{definition}[See \cite{Li}]\label{Def:7}
Let $r \geq 1$ and $\mathcal{H}= (V_\mathcal{H}, E_\mathcal{H})$ be a hypergraph with $u \in V_\mathcal{H}$ and $e_1, \ldots , e_r \in E_\mathcal{H}$ such that $u\notin e_i$ for $i= 1,\ldots,r$.
Suppose that $v_i\in e_i$ and write
$e_i'=(e_i\setminus\{v_i\})\cup \{u\}~(i=1,\ldots,r)$.
Let $\mathcal{H}'=(V_\mathcal{H}, E_{\mathcal{H}'})$ be the hypergraph with
$E_{\mathcal{H}'}=(E_{\mathcal{H}}\setminus\{e_1,\ldots,e_r\})\cup\{e_1',\ldots,e_r'\}$.
Then we say that $\mathcal{H}'$ is obtained from $\mathcal{H}$ by moving edges $(e_1, \ldots , e_r)$ from $(v_1,\ldots,v_r)$ to $u$.
\end{definition}

\begin{remark}[See \cite{Li}]
{\em
(a) The vertices $v_1,\ldots, v_r$ need not be distinct.

(b) The new hypergraph $\mathcal{H}'$ may contain multiple edges. But if $\mathcal{H}$ is acyclic and there is an edge $e\in E$ containing all the vertices $u, v_1,\ldots,v_r$,
then $\mathcal{H}'$ contains no multiple edges.}
\end{remark}
Let $\mathcal{H}=(V_\mathcal{H}, E_\mathcal{H})$ be a hypergraph with the non-increasing degree sequence $\pi(\mathcal{H})=(d_1,d_2,\ldots,d_n)$. If there exist two vertices $v_i$ and $v_j$ such that $d_i \geq d_j + 2$, where $v_i\in e\in E_\mathcal{H}$ and $v_j\notin e$, then we define a new hypergraph $\mathcal{H}'$, which is obtained from $\mathcal{H}$ by the operation of moving edge $e$ from $v_i$ to $v_j$ (see Definition \ref{Def:7}).
 Thus $\pi(\mathcal{H}')=(d_1, d_2,\ldots, d_{i-1}, d_i - 1, d_{i+1}, \ldots, d_{j-1}, d_j + 1, d_{j+1},\ldots, d_n)$. By modifying the proof of  \cite[Lemma 1]{Cao}, we obtain the following result.

\begin{lemma}\label{Lem:1}
Let $\mathcal{H}$ and $\mathcal{H}'$ be two hypergraphs specified as above. Then $h(\mathcal{H})>h(\mathcal{H}')$.
\end{lemma}

\begin{proof}
Note that $d_i \geq d_j + 2$, that is, $d_i>d_i-1\geq d_j + 1>d_j$. Then
\begin{eqnarray*}
h(\mathcal{H})- h(\mathcal{H}')
&=&d_i \log d_i +d_j \log d_j - (d_i - 1) \log (d_i - 1)-(d_j + 1) \log (d_j + 1)\\
&=&\Big(d_i \log d_i-(d_i - 1) \log (d_i - 1)\Big)-\Big((d_j + 1) \log (d_j + 1)-d_j \log d_j\Big)\\
&=&\left(\log\xi_{1}+\frac{1}{\ln2}\right)-\left(\log\xi_{2}+\frac{1}{\ln2}\right)\\
&>&0,
\end{eqnarray*}
where $\xi_{1}\in (d_i-1,d_i)$ and $\xi_{2}\in (d_j,d_j+1)$. This completes the proof.\qed
\end{proof}

The following \emph{edge-releasing }operation on linear hypergraphs is a special case of
the edge moving operation defined in Definition \ref{Def:7}.

\begin{definition}[See \cite{Li}]\label{Def:8}
Let $\mathcal{H}$ be a $k$-uniform linear hypergraph, $e$ be a non-pendent edge of $\mathcal{H}$
and $u\in e$. Let $e_1,e_2,\ldots,e_r$ be all the edges of $\mathcal{H}$ adjacent to $e$ but not containing $u$, and suppose that $e_i\cap e=\{v_i\}$ for $i=1,\ldots,r$. Let $\mathcal{H}'$ be the hypergraph obtained
from $\mathcal{H}$ by moving edges $(e_1,e_2,\ldots,e_r)$ from $(v_1,\ldots,v_r)$ to $u$. Then $\mathcal{H}'$ is said to be
obtained from $\mathcal{H}$ by an edge-releasing operation on $e$ at $u$.
\end{definition}

\begin{remark}[See \cite{Li}]\label{Rem:2}
{\em (a) In Definition \ref{Def:8}, the vertices $v_1,\ldots, v_r$ need not be distinct.

(b) If $\mathcal{H}$ is acyclic,
then $\mathcal{H}'$ contains no multiple edges.

(c) If $\mathcal{H}'$ and $\mathcal{H}''$ are two hypergraphs obtained
from a $k$-uniform linear hypergraph $\mathcal{H}$ by an edge-releasing operation on some edge
$e$ at $u$ and at $v$, respectively, then $\mathcal{H}'$ and $\mathcal{H}''$ are isomorphic.
}
\end{remark}

\begin{proposition}[See \cite{Li}]\label{Prop:4}
Let $\mathcal{H}'$ be a hypergraph obtained from a $k$-uniform supertree $\mathcal{H}$ by an edge-releasing operation on a non-pendent edge $e$ of $\mathcal{H}$. Then $\mathcal{H}'$ is also a supertree.
\end{proposition}

 \begin{proposition}\label{Prop:5}
Let $\mathcal{H}'$ be a hypergraph obtained from a unicyclic $k$-uniform hypergraph $\mathcal{H}$ by the operation of moving edge. If $\mathcal{H}'$ is connected, then $\mathcal{H}'$ is also a unicyclic $k$-uniform hypergraph.
\end{proposition}

\begin{proof}
By the definition of the operation of moving edge, we can see that $\mathcal{H}$ and $\mathcal{H}'$ the same number of edges. Thus, by Proposition \ref{Prop:0} we have
$\mid E_{\mathcal{H}}\mid=\mid E_{\mathcal{H}'}\mid=\frac{n}{k-1}$. Note that $\mathcal{H}'$ is connected. Hence, by Proposition \ref{Prop:2}, we conclude that $\mathcal{H}'$ is also a unicyclic $k$-uniform hypergraph. \qed
\end{proof}

\begin{proposition}\label{Prop:6}
Let $\mathcal{H}'$ be a hypergraph obtained from a bicyclic $k$-uniform hypergraph $\mathcal{H}$ by the operation of moving edge. If $\mathcal{H}'$ is connected, then $\mathcal{H}'$ is also a bicyclic $k$-uniform hypergraph.
\end{proposition}
\begin{proof}
By the definition of the operation of moving edge, we can see that $\mathcal{H}$ and $\mathcal{H}'$ the same number of edges. Thus, by Definitions \ref{Def:6} and \ref{Def:Hbicyc} we have
$\mid E_{\mathcal{H}}\mid=\mid E_{\mathcal{H}'}\mid=\frac{n+1}{k-1}$.  Note that $\mathcal{H}'$ is connected. Hence, by Proposition \ref{Prop:3}, we conclude that $\mathcal{H}'$ is also a bicyclic $k$-uniform hypergraph.
\end{proof}

\begin{lemma}\label{Lem:2}
Let $\mathcal{H}'$ be the hypergraph obtained from a connected $k$-uniform linear hypergraph $\mathcal{H}$ by an edge-releasing operation on $e$ at $u$, where $e$ is a non-pendent edge of $\mathcal{H}$ and $u\in e$. Then $h(\mathcal{H})< h(\mathcal{H}')$.
\end{lemma}
\begin{proof}
Suppose that $u'\in e$ and $u\not=u'$. By (c) of Remark \ref{Rem:2},  if $\mathcal{H}''$ is the hypergraph obtained from $\mathcal{H}$ by an edge-releasing operation on
$e$ at $u'$, then $\mathcal{H}''$ is isomorphic to $\mathcal{H}'$. So, without loss of generality, we assume that $d_u=\max_{v\in e}\{d_v\}$. Let $e_1,e_2,\ldots,e_r$ be all the edges of $\mathcal{H}$ adjacent to $e$ but not containing $u$, and suppose that $e_i\cap e=\{v_i\}$ for $i=1,\ldots,r$. Then $d_u\ge d_{v_i}>d_{v_i}-1$, and $d_{v_i}\geq2$ for $i=1,\ldots,r$. Suppose that $\pi(\mathcal{H})=(\ldots,d_u,\ldots,d_{v_1},\ldots,d_{v_2},\ldots,d_{v_r},\ldots)$ is the non-increasing degree sequence of $\mathcal{H}$. Let $\mathcal{H}_1$ be the hypergraph obtained from $\mathcal{H}$ by moving an edge $e_1$ from $v_1$ to $u$ (see Definition \ref{Def:7}). Then $\mathcal{H}_1=(V_\mathcal{H}, E_{\mathcal{H}_1})$ is a hypergraph with
$E_{\mathcal{H}_1}=(E_\mathcal{H}\setminus\{e_1\})\cup\{e_1'\}$, where $e_1'=(e_1\setminus\{v_1\})\cup \{u\}$.
Clearly, the degree sequence of $\mathcal{H}_1$ is
\[\pi(\mathcal{H}_1)=(\ldots,d_u+1,\ldots,d_{v_1}-1,\ldots,d_{v_2},\ldots,d_{v_r},\ldots).\]
Then we have
\begin{align*}
h(\mathcal{H})-h(\mathcal{H}_1)
&=d_u \log d_u +d_{v_1} \log d_{v_1} - (d_u + 1) \log (d_u + 1)-(d_{v_1}- 1) \log (d_{v_1} - 1)\\
&=-\Big((d_u + 1) \log (d_u + 1)-d_u \log d_u\Big)+\Big(d_{v_1} \log d_{v_1}-(d_{v_1}- 1) \log (d_{v_1} - 1)\Big)\\
&=-\left(\log\xi_{1}+\frac{1}{\ln2}\right)+\left(\log\xi_{2}+\frac{1}{\ln2}\right)\\
&<0,
\end{align*}
where $\xi_{1}\in (d_u,d_u+1)$ and $\xi_{2}\in (d_{v_1}-1,d_{v_1})$. So, $h(\mathcal{H})<h(\mathcal{H}_1)$.

%\begin{eqnarray*}
%h(\mathcal{H})-h(\mathcal{H}_1)
%&=&d_u \log d_u +d_{v_1} \log d_{v_1} - (d_u + 1) \log (d_u + 1)-(d_{v_1}- 1) \log (d_{v_1} - 1)\\
%&=&-\left[d_u \log \frac{1}{d_u} +d_{v_1} \log \frac{1}{d_{v_1}} - (d_u + 1) \log \frac{1}{d_u + 1}-(d_{v_1}- 1) \log \frac{1}{d_{v_1} - 1}\right]\\
%&=&-(d_u+d_{v_1})\bigg{[}\frac{d_u}{d_u+d_{v_1}} \log \frac{d_u+d_{v_1}}{d_u} +\frac{d_{v_1}}{d_u+d_{v_1}} \log \frac{d_u+d_{v_1}}{d_{v_1}}\\
%&&- \frac{d_u + 1}{d_u+d_{v_1}} \log \frac{d_u+d_{v_1}}{d_u + 1}
%-\frac{d_{v_1}- 1}{d_u+d_{v_1}} \log \frac{d_u+d_{v_1}}{d_{v_1} - 1}\bigg{]}\\
%&=&(d_u+d_{v_1})\left[g\left(\frac{d_{v_1}}{d_u+d_{v_1}}\right)-g\left(\frac{d_u + 1}{d_u+d_{v_1}}\right)\right]\\
%&=&(d_u+d_{v_1})\bigg{[}g\left(\frac{1}{2}+\frac{1-[d_u-(d_{v_1}-1)]}{2(d_u+d_{v_1})}\right)
%-g\left(\frac{1}{2}+\frac{1+[d_u-(d_{v_1}-1)]}{2(d_u+d_{v_1})}\right)\bigg{]},
%\end{eqnarray*}
%where $g(x)=x\log x+(1-x)\log(1-x)$, $x\in (0,1)$. Note that for $x\in (0,1)$, $g'(x)=(\ln x-\ln(1-x))\frac{1}{\ln2}$ and $g''(x)=(\frac{1}{x}+\frac{1}{1-x})\frac{1}{\ln2}>0$.
%Then $g'(x)=0$ for $x=\frac{1}{2}$ , so $g'(x) < 0$ for $x\in (0,\frac{1}{2})$ and $g'(x) > 0$ for $x\in (\frac{1}{2},1)$, and thus indeed $$g\left(\frac{1}{2}+\frac{1-[d_u-(d_{v_1}-1)]}{2(d_u+d_{v_1})}\right)
%<g\left(\frac{1}{2}+\frac{1+[d_u-(d_{v_1}-1)]}{2(d_u+d_{v_1})}\right).$$
Define $\mathcal{H}_{i+1}$ to be the hypergraph obtained from $\mathcal{H}_i$ by moving an edge $e_i$ from $v_i$ to $u$ for $i=1,2,\ldots,r-1$. Iteratively applying the technique used to prove $h(\mathcal{H})<h(\mathcal{H}_1)$, we obtain
$$h(\mathcal{H})<h(\mathcal{H}_1)<\cdots<h(\mathcal{H}_{r-1})<h(\mathcal{H}_{r}).$$
Clearly, $\mathcal{H}_{r}$ is isomorphic to $\mathcal{H}'$. This completes the proof.\qed

%where $\pi(\mathcal{H}')=(\cdots,d_u+r,\cdots,d_{v_1}-1,\cdots,d_{v_2}-1,\cdots,d_{v_r}-1,\cdots)$.

\end{proof}

%$\pi(\mathcal{H}_2)=(\cdots,d_u+2,\cdots,d_{v_1}-1,\cdots,d_{v_2}-1,\cdots,d_{v_r},\cdots)$.
%
%$$h(\mathcal{H}_1)-h(\mathcal{H}_2)<0.$$
%$$\vdots$$
%$$\vdots$$

%$$h(\mathcal{H}_{r-1})-h(\mathcal{H}')<0.$$
%Then we have
%$$h(\mathcal{H})-h(\mathcal{H}')<0.$$
%i.e.,
%$$h(\mathcal{H})< h(\mathcal{H}').$$

Let $\mathcal{H}=(V_\mathcal{H}, E_\mathcal{H})$ be a connected hypergraph with the non-increasing degree sequence $\pi(\mathcal{H})=(d_1,d_2,\ldots,d_n)$. If there exist two vertices $v_i$ and $v_j$ such that $d_i \geq d_j$, where $v_j\in e\in E_\mathcal{H}$ and $v_i\notin e$, then we define a new hypergraph $\mathcal{H}'$, which is obtained from $\mathcal{H}$ by the operation of moving edge $e$ from $v_j$ to $v_i$. Thus $\pi(\mathcal{H}')=(d_1, d_2,\ldots, d_{i-1}, d_i+1, d_{i+1}, \ldots, d_{j-1}, d_j- 1, d_{j+1},\ldots, d_n)$. By modifying the proof of $h(\mathcal{H})<h(\mathcal{H}_1)$ in Lemma \ref{Lem:2}, we obtain the following result.

\begin{lemma}\label{Lem:3}
Let $\mathcal{H}$ and $\mathcal{H}'$ be as above.
Then $h(\mathcal{H})< h(\mathcal{H}')$.
\end{lemma}
%\begin{proof}
%If $d_i>d_j-1$, $\pi(\mathcal{H})=(\cdots,d_i,\cdots,d_j,\cdots)$, where $d_i\geq2$.
%
%$\pi(\mathcal{H}')=(\cdots,d_i+1,\cdots,d_j-1,\cdots)$.
%
%Then
%\begin{eqnarray*}
%h(\mathcal{H})-h(\mathcal{H}')
%&=&d_i \log d_i +d_j \log d_j - (d_i + 1) \log (d_i + 1)-(d_j- 1) \log (d_j - 1)\\
%&<&0.
%\end{eqnarray*}
%Then we have
%$$h(\mathcal{H})< h(\mathcal{H}').$$
%This completes the proof.
%\end{proof}
%\begin{lemma}
% Let $\mathcal{T}'$ be the supertree obtained from $\mathcal{T}$ by edge-releasing the non-pendent edge $e$ of $\mathcal{T}$.
%Then we have $h(\mathcal{T})\leq h(\mathcal{T}')$.
%\end{lemma}
\section{Extremality of $I_d^1(\mathcal{H})$ among all $k$-uniform ($k\geq3$) supertrees}

In this section, we investigate the extremality of  $I_d^1(\mathcal{H})$ among all $k$-uniform ($k\geq3$) supertrees.
\begin{lemma}\label{Lem:4}
 Let $\mathcal{T}$ be a $k$-uniform ($k\geq3$) supertree on $n$ vertices with $m=\frac{n-1}{k-1}\ge2$ edges. Then

\begin{itemize}
  \item[\rm (a)] $h(\mathcal{T})\geq 2(m-1)\log2$, with the equality holding if and only if $\mathcal{T}\in \mathcal{T}^{\ast}$, where $\mathcal{T}^{\ast}$ denotes the family of all $k$-uniform ($k\geq3$) supertrees on $n$ vertices with $m=\frac{n-1}{k-1}\ge2$ edges whose maximum degree is $2$; and
  \item[\rm (b)] $h(\mathcal{T})\leq m\log m$, with the equality holding if and only if $\mathcal{T}\cong S_{m+1}^{k}$.
\end{itemize}
\end{lemma}

\begin{proof} We prove (a) by contradiction. Suppose that $\mathcal{T}\notin \mathcal{T}^{\ast}$ attains the minimum value among all $k$-uniform ($k\geq3$) supertrees on $n$ vertices with $m=\frac{n-1}{k-1}\ge2$ edges. Then, there exists at least one vertex $u$ in $\mathcal{T}$ with $d_u\geq 3$. Let $e$ denote a non-pendent edge containing $u$, $e_1\not=e$ denote another edge containing $u$, and $e_0$ denote a pendent edge with $u\notin e_0$. Suppose that $u_s\in e_0$ and $d_{u_s}=1$.
Let $\mathcal{T}'=(V_{\mathcal{T}},E_{\mathcal{T}'})$ be the hypergraph with $E_{\mathcal{T}'}=(E_{\mathcal{T}}\setminus \{e_1\})\cup \{e_1'\}$, where $e_1'=(e_1\setminus \{u\})\cup\{u_s\}$. Then $\pi(\mathcal{T}) = (\ldots,d_{u},\ldots,d_{u_s},\ldots)$, and
$\pi(\mathcal{T}')=(\ldots,d_{u}-1,\ldots,d_{u_s}+1,\ldots)$. Note that $d_u\geq3=d_{u_s}+2$. Then, by Lemma \ref{Lem:1}, we have $h(\mathcal{T}) > h(\mathcal{T}')$, a contradiction.
This implies that only $\mathcal{T}\in \mathcal{T}^{\ast}$ can attain the minimum value among all $k$-uniform ($k\geq3$) supertrees on $n$ vertices with $m=\frac{n-1}{k-1}\ge2$ edges. By simple computation, we have $h(\mathcal{T})=2(m-1)\log2$ for any $\mathcal{T}\in \mathcal{T}^{\ast}$. This completes the proof of (a).

%We can prove that if $\mathcal{T}$ is a $k$-uniform supertree on $n$
%vertices (with $m=n'-1$ edges, where $n'=\frac{n-1}{k-1}+1$). Then $h(\mathcal{T})\geq h(\mathcal{T}^{\ast})$, the equality holds if and only if $\mathcal{T}\cong \mathcal{T}^{\ast}$.

Next, we prove (b) by the induction. Denote by $N_2(\mathcal{T})$ the number of non-pendent vertices of  $\mathcal{T}$.  If $\mathcal{T}$ is a $k$-uniform ($k\geq3$) supertree on $n$ vertices and $m=\frac{n-1}{k-1}\ge2$ edges with $N_2(\mathcal{T})=1$, then $\mathcal{T}\cong S_{m+1}^{k}$ and then $h(\mathcal{T})=m\log m$. If $\mathcal{T}$ is a $k$-uniform supertree on $n$ vertices and $m=\frac{n-1}{k-1}\ge2$ edges with $N_2(\mathcal{T})=2$, then $\mathcal{T}\ncong S_{m+1}^{k}$. In such case, $m\ge3$ and $\mathcal{T}$ has only two vertices (say $u$ and $v$) whose degrees are not less than  $2$. Suppose that $e$ is the non-pendent edge containing $u$ and $v$. By repeating the edge-releasing operation on $e$ at $u$ (see Definition \ref{Def:8}), we obtain $S_{m+1}^{k}$. Then, by  Lemma \ref{Lem:2}, we have $h(\mathcal{T})<h(S_{m+1}^{k})=m\log m$.

% \hd{Assume that $u\in e$. Let
% $\mathcal{T}''$ be the hypergraph obtained from $\mathcal{T}'$ by an edge-releasing operation on $e$ at $u$ (see Definition \ref{Def:8}). Then, by Proposition \ref{Prop:4} and Lemma \ref{Lem:2},Assume that $u\in e$. Let
% $\mathcal{T}''$ be the hypergraph obtained from $\mathcal{T}'$ by an edge-releasing operation on $e$ at $u$ (see Definition \ref{Def:8}). Then, by Proposition \ref{Prop:4} and Lemma \ref{Lem:2},Assume that $u\in e$. Let
% $\mathcal{T}''$ be the hypergraph obtained from $\mathcal{T}'$ by an edge-releasing operation on $e$ at $u$ (see Definition \ref{Def:8}). Then, by Proposition \ref{Prop:4} and Lemma \ref{Lem:2},Assume that $u\in e$. Let
% $\mathcal{T}''$ be the hypergraph obtained from $\mathcal{T}'$ by an edge-releasing operation on $e$ at $u$ (see Definition \ref{Def:8}). Then, by Proposition \ref{Prop:4} and Lemma \ref{Lem:2},}

 Assume that $h(\mathcal{T})< m\log m$ for all $k$-uniform ($k\geq3$) supertree $\mathcal{T}$ on $n$ vertices and $m=\frac{n-1}{k-1}\ge2$ edges with $1<N_2(\mathcal{T})\le l\geq2$. For a $k$-uniform ($k\geq3$) supertree $\mathcal{T}'$ on $n$ vertices and $m=\frac{n-1}{k-1}\ge2$ edges with $N_2(\mathcal{T}')=l+1$, it is easy to see that $\mathcal{T}'\ncong S_{m+1}^{k}$. Then, there exists one non-pendent edge $e$ in $\mathcal{T}'$.
 Assume that $u\in e$. Let
 $\mathcal{T}''$ be the hypergraph obtained from $\mathcal{T}'$ by an edge-releasing operation on $e$ at $u$ (see Definition \ref{Def:8}). Then, by Proposition \ref{Prop:4} and Lemma \ref{Lem:2},
 %( If there exists $u\in e$ such that $d_u> \max \{d_{v_i}\}-1\geq1$, then by Lemma 4, we have $h(\mathcal{T})< h(\mathcal{T}')$; otherwise, let $u\triangleq v_i$ whose degree is $\max \{d_{v_j}\}, j=1,\ldots, r$, then by Lemma 4, we have $h(\mathcal{T})< h(\mathcal{T}')$). and Lemma \ref{Lem:2}and $h(\mathcal{T}')< h(\mathcal{T}'')$h(\mathcal{T}')<
we obtain that $\mathcal{T}''$ is a $k$-uniform ($k\geq3$) supertree and $h(\mathcal{T}')< h(\mathcal{T}'')$. Note that
$N_2(\mathcal{T}'')< N_2(\mathcal{T}')=l+1$. Based on the assumption that $h(\mathcal{T''})< m\log m$ for all $k$-uniform ($k\geq3$) supertree $\mathcal{T''}$ on $n$ vertices and $m=\frac{n-1}{k-1}\ge2$ edges with $1<N_2(\mathcal{T''})\le l\geq2$, we have $h(\mathcal{T}')< h(\mathcal{T}'')<m\log m$. Thus, $h(\mathcal{T})< m\log m$ for all $k$-uniform ($k\geq3$) supertree $\mathcal{T}$ on $n$ vertices and $m=\frac{n-1}{k-1}\ge2$ edges with $N_2(\mathcal{T})\ge2$.

From the discussion above, we note that only $\mathcal{T}\cong S_{m+1}^{k}$ can attain the maximum value among all $k$-uniform ($k\geq3$) supertrees on $n$ vertices with $m=\frac{n-1}{k-1}\ge2$ edges, that is, $h(\mathcal{T})=m\log m$ if and only if $\mathcal{T} \cong S_{m+1}^{k}$. This completes the proof of (b).\qed
\end{proof}

From Lemma \ref{Lem:4} and Equality (\ref{Equ:3.1}), we get the following result.

\begin{theorem}\label{The:1}
 Let $\mathcal{T}$ be a $k$-uniform ($k\geq3$) supertree on $n$ vertices with $m=\frac{n-1}{k-1}\ge2$ edges and $\mathcal{T}^{\ast}$ as in Lemma \ref{Lem:4}.  Then
  $$
  \log(km)-\frac{\log m}{k}\le I_d^1(\mathcal{T})\leq\log(km)-\frac{2(m-1)\log2}{km},
  $$
where the first equality holds if and only if $\mathcal{T}\cong S_{m+1}^{k}$, and the second equality holds if and only if $\mathcal{T}\in \mathcal{T}^{\ast}$.
\end{theorem}

\section{Extremality of $I_d^1(\mathcal{H})$ among all unicyclic $k$-uniform hypergraphs}
In this section, we investigate the extremality of $I_d^1(\mathcal{H})$ among all unicyclic $k$-uniform ($k\geq3$) hypergraphs.
\begin{lemma}\label{Lem:5}
 Let $\mathcal{H}$ be a unicyclic $k$-uniform ($k\geq3$) hypergraph on $n$ vertices with $m=\frac{n}{k-1}\ge2$ edges. Then %(with $m=n'$ edges, where $n'=\frac{n}{k-1}$).

\begin{itemize}
  \item[\rm (a)] $h(\mathcal{H})\geq 2m\log2$, with the equality holding if and only if $\mathcal{H}\in \mathcal{H}^{I}$, where $\mathcal{H}^{I}$ denotes the family of all unicyclic $k$-uniform ($k\geq3$) hypergraphs on $n$ vertices with $m=\frac{n}{k-1}\ge2$ edges whose maximum degree is $2$; and
  \item[\rm (b)] $h(\mathcal{H})\leq m\log m+2\log 2$, with the equality holding if and only if $\mathcal{H}\cong \mathcal{H}^{II}$, where $\mathcal{H}^{II}$ is shown in Figure \ref{unicyclic}.
\end{itemize}
\end{lemma}

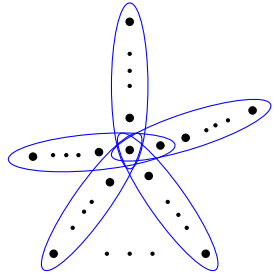
\begin{figure}[h]
\centering
\vspace{-0.35cm}
\begin{tikzpicture}[scale=2.0]
\draw[rotate around={55:(0,0)},blue] (0,0) ellipse (0.55 and 0.12);
\filldraw [black]
%(-0.075,0.769)circle (1.3pt)
%(-0.125,0.689)circle (1.3pt)
%(-0.025,0.849)circle (1.3pt)
%center
(0.2500,0.34380)circle (0.7pt)
 (0.2500,0.9816)circle (0.3pt)
(0.2500,0.5564)circle (0.7pt)
(0.2500,1.1942)circle (0.7pt)
    (0.250,0.769)circle (0.3pt)
    (0.250,0.869)circle (0.3pt)
  %zuoxiao
    (0.12,0.13)circle (0.7pt)
    (0,0) circle (0.3pt)
    (-0.05,-0.065) circle (0.3pt)
     %zuohengxiao
     (0.0480,0.33)circle (0.7pt)
     (-0.087,0.31)circle (0.3pt)
     (-0.167,0.31)circle (0.3pt)
   %youxiao
    (0.3750,0.1719)circle (0.7pt)
    (0.5,0)circle (0.3pt)
    (0.57,-0.0859)circle (0.3pt)
    %youhengxiao
    (0.618,0.425)circle (0.7pt)

    (0.754,0.475)circle (0.3pt)
    (0.814,0.5075)circle (0.3pt)
(0.8960,0.5406)circle (0.3pt)
(0.4520,0.3738)circle (0.7pt)

    ( -0.2500,-0.3438)circle (0.7pt)
    (0.2500,-0.3438)circle (0.3pt)
     (0.1,-0.3438)circle (0.3pt)
   (0.40,-0.3438)circle (0.3pt)

   (0.7500,-0.3438)circle (0.7pt)
  ( 1.0580,0.6062)circle (0.7pt)

    %four
    (-0.1250,-0.1719)circle (0.3pt)
    (0.6250,-0.1719)circle (0.3pt)

  (-0.2540,0.31)circle (0.3pt)
   (-0.3860,0.3)circle (0.7pt);
 %B
\draw[rotate around={125:(0.5,0)},blue] (0.5,0) ellipse (0.55 and 0.12);
\draw[rotate around={18:(0.654,0.475)},blue] (0.654,0.475) ellipse (0.55 and 0.12);
\draw[rotate around={90:(0.250,0.769)},blue] (0.250,0.769) ellipse (0.55 and 0.12);
\draw[rotate around={5:(0.0000,0.3282 )},blue] (0.0000,0.3282) ellipse (0.55 and 0.12);
\end{tikzpicture}
\caption{The unicyclic $k$-uniform hypergraph $\mathcal{H}^{II}$}\label{unicyclic}
\end{figure}

\begin{proof}
 We prove (a) by contradiction. Suppose that $\mathcal{H}\notin \mathcal{H}^{I}$ attains the minimum value among all unicyclic $k$-uniform ($k\geq3$) hypergraphs on $n$ vertices with $m=\frac{n}{k-1}\ge2$ edges. Then, there exists at least one vertex $u$ with $d_u\geq 3$.
Let $e$ denote a non-pendent edge containing $u$ and $C=v_1e_1v_2e_2\cdots e_tv_{t+1}~(v_{t+1}=v_1)$  the unique cycle in $\mathcal{H}$. In the following, we say $u\in C$ if $u\in \cup_{i=1}^t e_i$; otherwise, $u\notin C$, and $f\in C$ if $f=e_i$ for some $i\in\{1,2,\ldots,t\}$; otherwise, $f\notin C$.

 \noindent{\bf Case 1.} $u\in C$.

  Without loss of generality, we assume $e=e_1$ and $u\in e_1$. Find a longest path $P=uf_1u_1f_2\cdots f_su_s$ starting at $u$ such that $u_i\notin C$ and $f_i\notin C$ for $i=1,2,\ldots,s$. Obviously, $d_{u_s}=1$.

\noindent{\bf Subcase 1.1} $u=v_1$ or $v_2$.

Without loss of generality, we assume $u=v_1$. Then $d_{v_1}\geq3$.
 Let $\mathcal{H}'=(V_{\mathcal{H}},E_{\mathcal{H}'})$ be the hypergraph with $E_{\mathcal{H}'}=(E_{\mathcal{H}}\setminus \{e_1\})\cup \{e_1'\}$, where $e_1'=(e_1\setminus\{v_1\})\cup\{u_s\}$. Then, by Proposition \ref{Prop:5}, $\mathcal{H}'$ is also a unicyclic $k$-uniform hypergraph. Note that $\pi(\mathcal{H})=(\ldots, d_{v_1}, \ldots, d_{u_s},\ldots)$,
$\pi(\mathcal{H}')=(\ldots, d_{v_1}-1, \ldots, d_{u_s}+1,\ldots)$, and $d_{v_1}\geq 3=d_{u_s}+2$. Then, by Lemma \ref{Lem:1}, we have $h(\mathcal{H})>h(\mathcal{H}')$, a contradiction.

\noindent{\bf Subcase 1.2} $u\neq v_1,v_2$.

Let $f_1\not=e_1$ denote another edge containing $u$.
Let $\mathcal{H}'=(V_{\mathcal{H}},E_{\mathcal{H}'})$ be the hypergraph with $E_{\mathcal{H}'}=(E_{\mathcal{H}}\setminus \{e_1, f_1\})\cup \{e_1',f_1'\}$, where $e_1'=(e_1\setminus\{v_1\})\cup\{u_s\}$ and $f_1'=(f_1\setminus\{u\})\cup\{v_1\}$. Then, by Proposition \ref{Prop:5}, $\mathcal{H}'$ is also a unicyclic $k$-uniform hypergraph. Note that $\pi(\mathcal{H})=(\ldots, d_{u}, \ldots, d_{u_s},\ldots)$,$\pi(\mathcal{H}')=(\ldots, d_u-1, \ldots, d_{u_s}+1,\ldots)$, and $d_{u}\geq 3=d_{u_s}+2$. Then, by Lemma \ref{Lem:1}, we get $h(\mathcal{H})>h(\mathcal{H}')$, a contradiction.

\noindent{\bf Case 2.} $u\notin C$.

Find a longest path $Q=vf_1u_1f_2\cdots f_su_s$ ($u_i\notin C$ and $f_i\notin C$ for $i=1,2,\ldots,s$) starting at $v$ such that $v\in e_i$ for some $i\in\{1,2,\ldots,t\}$ and $f_{j-1}\cap f_j=\{u\}$ for some $j\in\{1,2,\ldots,s\}$.
Let $\mathcal{H}'=(V_{\mathcal{H}},E_{\mathcal{H}'})$ be the hypergraph with $E_{\mathcal{H}'}=(E_{\mathcal{H}}\setminus \{e_i, f_j\})\cup \{e_i',f_j'\}$, where $e_i'=(e_i\setminus\{v_i\})\cup\{u_s\}$ and $f_j'=(f_j\setminus\{u\})\cup\{v_i\}$. Then, by Proposition \ref{Prop:5}, $\mathcal{H}'$ is also a unicyclic $k$-uniform hypergraph. Note that $\pi(\mathcal{H})=(\ldots,d_u,\ldots,d_{u_s},\ldots)$, $\pi(\mathcal{H}')=(\ldots,d_u-1,\ldots,d_{u_s}+1,\ldots)$, and $d_u\geq 3=d_{u_s}+2$. Then, by Lemma \ref{Lem:1}, we have $h(\mathcal{H})>h(\mathcal{H}')$, a contradiction.

Thus, the discussions above imply that only $\mathcal{H}\in \mathcal{H}^{\uppercase\expandafter{\romannumeral1}}$ can attain the minimum value among all unicyclic $k$-uniform ($k\geq3$) hypergraphs on $n$ vertices with $m=\frac{n}{k-1}\ge2$ edges. By simple computation, we have $h(\mathcal{H})= 2m\log2$ for any $\mathcal{H}\in \mathcal{H}^{I}$. This completes the proof of (a). \medskip

%From the above three claims, we can easily obtain that $h(\mathcal{H})$ attains the minimum value if and only if $\mathcal{H}\cong \mathcal{H}^{\ast}$.\qed
%\hd{If $d_{v_1}=2$ and $d_{u}=2$. Let $f_1'=(f_1\setminus\{u\})\cup\{v_1\}$.
% $\mathcal{H}'=\mathcal{H}-f_1+f_1'$, then $\pi(\mathcal{H})=(d_1, d_2,\ldots, d_{v_1}, \ldots, d_u,\ldots, d_n)$. $\mathcal{H}'$ be the graph obtained from $\mathcal{H}$ by replacing
%the pair $(d_{v_1}, d_u)$ by the pair $(d_{v_1}+1, d_u-1)$, that is, $\pi(\mathcal{H}')=(d_1, d_2,\ldots, d_{v_1}+1,, \ldots, d_u-1,\ldots, d_n)$. By Lemma 2, we have $h(\mathcal{H})>h(\mathcal{H}'),$
%which contradicts to the choice of $\mathcal{H}$.}

Next, we give the proof of (b). Suppose that $\mathcal{H}$ attains the maximum value among all unicyclic $k$-uniform ($k\geq3$) hypergraphs on $n$ vertices with $m=\frac{n}{k-1}\ge2$ edges. Again, we use $C=v_1e_1v_2e_2\cdots e_tv_{t+1}~(v_{t+1}=v_1)$ to denote the unique cycle in $\mathcal{H}$.
\medskip

\noindent{\bf Claim 1.} \emph{There exists one vertex $v_i\in C$ with the maximum degree $\Delta(\mathcal{H})$.}
\medskip
%find the longest path $Q=vf_1u_1f_2\cdots f_su_s$ ($u_i\notin C$ and $f_i\notin C$ for $i=1,2,\ldots,s$) starting at $v$ such that $v\in e_i$ for some $i\in\{1,2,\ldots,t\}$ and $f_{j-1}\cap f_j=\{u\}$ for some $j\in\{1,2,\ldots,s\}$.

Suppose that there is no vertex on the cycle $C$ possessing the maximum degree $\Delta(\mathcal{H})$. Choose a vertex $v\notin C$ such that $d(v)=\Delta(\mathcal{H})$. Find a path $P_1=uf_1u_1f_2\cdots u_{s-1}f_sv$ ($u_i\notin C$ for $i=1,2,\ldots,s-1$ and $f_j\notin C$ for $j=1,2,\ldots,s$)
 starting at $u$ such that $u\in e_i$ for some $i\in\{1,2,\ldots,t\}$.
Let $\mathcal{H}'=(V_{\mathcal{H}},E_{\mathcal{H}'})$ be the hypergraph with $E_{\mathcal{H}'}=(E_{\mathcal{H}}\setminus \{e_i\})\cup \{e_i'\}$, where $e_i'=(e_i\setminus \{u\})\cup\{v\}$. By Proposition \ref{Prop:5}, $\mathcal{H}'$ is also a unicyclic $k$-uniform hypergraph. Note that  $\pi(\mathcal{H}) = (d_v,\ldots,d_u,\ldots)$,
$\pi(\mathcal{H}')=(d_v+1,\ldots,d_u-1,\ldots)$, and $d_v>d_u-1\ge1$. By Lemma \ref{Lem:3}, we obtain that $h(\mathcal{H})< h(\mathcal{H}')$, a contradiction.
\medskip

\noindent{\bf Claim 2.} \emph{Let $e(C)=\{e ~|~ e \text{ is an edge} \in C\}$. Then $|e(C)|=2.$}
\medskip

Without loss of generality, suppose that $d(v_0)=\Delta(\mathcal{H})$ and $v_2\not=v_0\in e_1$ ($v_0$ can be equal to $v_1$). If $|e(C)|\ge 3$, then
let $\mathcal{H}'=(V_{\mathcal{H}},E_{\mathcal{H}'})$ be the hypergraph with $E_{\mathcal{H}'}=(E_{\mathcal{H}}\setminus \{e_2\})\cup \{e_2'\}$, where $e_2'=(e_2\setminus \{v_2\})\cup\{v_0\}$. By Proposition \ref{Prop:5}, $\mathcal{H}'$ is also a unicyclic $k$-uniform hypergraph. Note that $\pi(\mathcal{H}) = (d_{v_0},\ldots,d_{v_2},\ldots)$, $\pi(\mathcal{H}')=(d_{v_0}+1,\ldots,d_{v_2}-1,\ldots)$, and $d_{v_0}\ge d_{v_2} >d_{v_2}-1\ge1$. By Lemma \ref{Lem:3}, we have $h(\mathcal{H})< h(\mathcal{H}')$, a contradiction.

\medskip

%If $|e(C)|\geq4$, then let $\mathcal{H}'=(V_{\mathcal{H}},E_{\mathcal{H}'})$ be the hypergraph with $E_{\mathcal{H}'}=(E_{\mathcal{H}}\setminus \{e_3\})\cup \{e_3'\}$, where $e_3'=(e_3\setminus \{v_4\})\cup\{v_0\}$. By Proposition \ref{Prop:5}, $\mathcal{H}'$ is also a $k$-uniform unicyclic hypergraph. Note that $\pi(\mathcal{H}) = (d_{v_0},\ldots,d_{v_4},\ldots)$,  $\pi(\mathcal{H}')=(d_{v_0}+1,\ldots,d_{v_4}-1,\ldots)$, and $d_{v_0}>d_{v_4}-1\ge1$.  By Lemma \ref{Lem:3}, we have $h(\mathcal{H})< h(\mathcal{H}')$, a contradiction.

\noindent{\bf Claim 3.} \emph{$\Delta(\mathcal{H})=m.$}
\medskip

If $\Delta(\mathcal{H})\leq m-1$, without loss of generality,  suppose that $d(v_0)=\Delta(\mathcal{H})$. Then, there exists at least one edge $e\in E(\mathcal{H})$ with $v_0\notin e$.
Let $e_i$ for some $i\in\{1,2,\ldots,t\}$ be the edge which adjacent to $e$ and $e_i\cap e=\{w\}$ and $\mathcal{H}'=(V_{\mathcal{H}},E_{\mathcal{H}'})$ be the hypergraph with $E_{\mathcal{H}'}=(E_{\mathcal{H}}\setminus \{e\})\cup \{e'\}$, where $e'=(e\setminus \{w\})\cup\{v_0\}$. By Proposition \ref{Prop:5}, $\mathcal{H}'$ is also a unicyclic $k$-uniform hypergraph. Note that $\pi(\mathcal{H}) = (d_{v_0},\ldots,d_w,\ldots)$,  $\pi(\mathcal{H}')=(d_{v_0}+1,\ldots,d_w-1,\ldots)$, and $d_{v_0}>d_w-1\ge1$. By Lemma \ref{Lem:3}, we obtain that $h(\mathcal{H})< h(\mathcal{H}')$, a contradiction.\medskip

From Claims 1-3, we obtain that $h(\mathcal{H})$ attains the maximum value among all unicyclic $k$-uniform ($k\geq3$) hypergraphs on $n$ vertices with $m=\frac{n}{k-1}\ge2$ edges if and only if $\mathcal{H}\cong \mathcal{H}^{\uppercase\expandafter{\romannumeral2}}$. By simple computation, we have $h(\mathcal{H})= m\log m+2\log 2$ for $\mathcal{H}\cong \mathcal{H}^{II}$. This completes the proof of (b).\qed
\end{proof}

From Lemma \ref{Lem:5} and Equality (\ref{Equ:3.1}), we have the following result.

\begin{theorem}\label{The:2}
 Let $\mathcal{H}$ be a unicyclic $k$-uniform ($k\geq3$) hypergraph on $n$ vertices with $m=\frac{n}{k-1}\ge2$ edges, and let $\mathcal{H}^{I}$, $\mathcal{H}^{II}$ be as in Lemma \ref{Lem:5}. Then
 $$
\log(km)-\frac{m\log m+2\log 2}{km}\le I_d^1(\mathcal{H})\leq \log(km)-\frac{2\log2}{k},
 $$
where the first equality holds if and only if $\mathcal{H}\cong  \mathcal{H}^{II}$, and the second equality holds if and only if $\mathcal{H}\in \mathcal{H}^{I}$.
\end{theorem}

\section{Extremality of $I_d^1(\mathcal{H})$ among all bicyclic $k$-uniform ($k\geq3$) hypergraphs}

In this section, we investigate the extremality of $I_d^1(\mathcal{H})$ among all bicyclic $k$-uniform ($k\geq3$) hypergraphs.
\begin{lemma}\label{Lem:6}
 Let $\mathcal{H}$ be a bicyclic $k$-uniform ($k\geq3$) hypergraph on $n$ vertices with $m=\frac{n+1}{k-1}\ge2$ edges. Then %(with $m=n'$ edges, where $n'=\frac{n}{k-1}$).

\begin{itemize}
  \item[\rm (a)]$h(\mathcal{H})\geq 2(m+1)\log2$, with the equality holding if and only if $\mathcal{H}\in \mathcal{H}^{III}$, where $\mathcal{H}^{III}$ denotes the family of all bicyclic $k$-uniform ($k\geq3$) hypergraphs on $n$ vertices with $m=\frac{n+1}{k-1}\ge2$ edges whose maximum degree is $2$; and
  \item[\rm (b)] $h(\mathcal{H})\leq m\log m+4\log 2$, with the equality holding if and only if $\mathcal{H}\cong \mathcal{H}^{IV}$ or $\mathcal{H}\cong \mathcal{H}^{V}$, where $\mathcal{H}^{IV}$ and $\mathcal{H}^{V}$ are shown in Figure \ref{bicyclic}.
\end{itemize}
\end{lemma}

 \begin{figure}[h]
 \subfloat[$\mathcal{H}^{IV}$]{
\label{DirSample_frac01}
  \begin{minipage}[c]{0.4\textwidth}
\centering
\begin{tikzpicture}[scale=2.0]
\draw[rotate around={55:(0,0)},blue] (0,0) ellipse (0.55 and 0.12);
\filldraw [black]
    %shengluehao
(0.1,-0.3438)circle (0.3pt)
(0.25,-0.3438)circle (0.3pt)
(0.4,-0.3438)circle (0.3pt)
    %zuoxiao
    (0,0) circle (0.3pt)
    (0.12,0.13)circle (0.7pt)
    (-0.05,-0.065)circle (0.3pt)
    ( -0.2500,   -0.3438)circle (0.7pt)
    (-0.1250,   -0.1719)circle (0.3pt)
    %youxiao
    (0.3750,0.1719)circle (0.7pt)
    (0.5,0)circle (0.3pt)
    (0.57,-0.085950)circle (0.3pt)
    ( 0.7500 ,  -0.3438)circle (0.7pt)
    (0.6250,   -0.1719)circle (0.3pt)
    %xiangjiaosandian
    (0.25000 ,0.34380)circle (0.7pt)
    (0.4220,0.3894)circle (0.7pt)
    (0.53,0.405)circle (0.7pt)
    (0.7,0.4693)circle (0.7pt)
    (0.842,0.52145)circle (0.3pt)
    (0.90,0.535)circle (0.3pt)
    (0.9760,    0.5536)circle (0.3pt)
    ( 1.1080,    0.59062)circle (0.7pt)
    %shangxiao
    (0.250,0.769)circle (0.3pt)
    (0.250,0.8627)circle (0.3pt)
    (0.2500,0.5564)circle (0.7pt)
    (0.2500,    0.9816)circle (0.3pt)
    (0.2500 ,   1.1942)circle (0.7pt)
  %zuohengda
   (-0.35540,    0.245)circle (0.7pt)
   (-0.230 ,   0.26)circle (0.3pt)
   (-0.09 ,   0.28)circle (0.3pt)
   (-0.15 ,   0.27)circle (0.3pt)
   (0.03 ,   0.31)circle (0.7pt);
 %B
\draw[rotate around={125:(0.5,0)},blue] (0.5,0) ellipse (0.55 and 0.12);
\draw[rotate around={15:(0.654,0.475)},blue] (0.654,0.475) ellipse (0.55 and 0.12);
\draw[rotate around={90:(0.250,0.769)},blue] (0.250,0.769) ellipse (0.55 and 0.12);
\draw[rotate around={8:(0.2500 ,   0.3438 )},blue] ( 0.0800,0.3438) ellipse (0.55 and 0.12);
\end{tikzpicture}
\end{minipage}
}
\subfloat[$\mathcal{H}^{V}$]{
\label{DirSample_frac02}
  \begin{minipage}[c]{0.40\textwidth}
\centering
\begin{tikzpicture}[scale=2.0]
\draw[rotate around={55:(0,0)},blue] (0,0) ellipse (0.55 and 0.12);
\filldraw [black]
%zuoxiao
    (0,0) circle (0.3pt)
    (-0.06,-0.0900) circle (0.3pt)
    (0.12,0.13)circle (0.7pt)
    ( -0.2500,-0.3438)circle (0.7pt)
    (-0.1250,-0.1719)circle (0.3pt)
    %youxiao
    (0.3250,0.1919)circle (0.7pt)
   (0.4375,0.08595)circle (0.7pt)
    (0.5,0)circle (0.3pt)
    (0.563,-0.08595)circle (0.3pt)
    ( 0.7500,-0.3438)circle (0.7pt)
    (0.6250,-0.1719)circle (0.3pt)
 %center
(0.25000 ,0.34380)circle (0.7pt)
 %%shengluehao
%(0.4,0.769)circle (0.3pt)
%(0.6,0.729)circle (0.3pt)
%(0.5,0.749)circle (0.3pt)
%shengluehao
(0.1,-0.3438)circle (0.3pt)
(0.25,-0.3438)circle (0.3pt)
(0.4,-0.3438)circle (0.3pt)
    %hengyou
    (0.4520,0.39094)circle (0.7pt)
    (0.8560,0.5406)circle (0.3pt)
    (0.624,0.435)circle (0.7pt)
    (0.75,0.49075)circle (0.3pt)
    (0.954,0.58)circle (0.3pt)
     ( 1.0580,0.6062)circle (0.7pt)
   %hengzuo
   (-0.1840,    0.2726)circle (0.3pt)
   (0.0480 ,   0.33)circle (0.7pt)
   (-0.270 ,   0.25)circle (0.3pt)
   (-0.10 ,   0.29)circle (0.3pt)
   (-0.3960 ,   0.2370)circle (0.7pt)
   %shangdian
   (0.250,    0.5557)circle (0.7pt)
   (0.2250,    0.657)circle (0.3pt)
   (0.19  ,  0.7576)circle (0.3pt)
   (0.160,    0.8595)circle (0.3pt)
   (0.13,    1.0)circle (0.7pt);
 %B
\draw[rotate around={125:(0.5,0)},blue] (0.5,0) ellipse (0.55 and 0.12);
\draw[rotate around={18:(0.654,0.475)},blue] (0.654,0.475) ellipse (0.55 and 0.12);
\draw[rotate around={102:(0.1250,0.5157)},blue] (0.20,0.40) ellipse (0.55 and 0.12);
\draw[rotate around={10:(0.00,0.3082)},blue] (0.00,0.3082) ellipse (0.55 and 0.12);
\end{tikzpicture}
\end{minipage}
}
\caption{The bicyclic $k$-uniform hypergraphs $\mathcal{H}^{IV}$ and $\mathcal{H}^{V}$.}\label{bicyclic}
\end{figure}
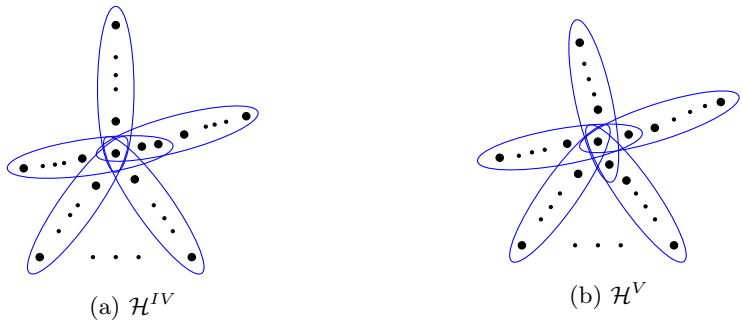

 \begin{proof}We prove (a) by contradiction. Suppose that $\mathcal{H}\notin \mathcal{H}^{III}$ attains the minimum value among all bicyclic $k$-uniform ($k\geq3$) hypergraphs on $n$ vertices with $m=\frac{n+1}{k-1}\ge2$ edges. Then, there exists at least one vertex $u$ with $d_u\geq 3$.
Let $e$ denote a non-pendent edge containing $u$ and
   $C=v_1e_1v_2e_2\cdots e_tv_{t+1}~(v_{t+1}=v_1)$ be a cycle in $\mathcal{H}$.

    \noindent{\bf Case 1.} $u\in C$.

 Without loss of generality, we assume $e=e_1$ and $u\in e_1$.
Find a longest path $P=uf_1u_1f_2\cdots f_su_s$ starting at $u$ such that $u_i\notin C$ and $f_i\notin C$ for $i=1,2,\ldots,s$.
Obviously, $d_{u_s}=1$.

\noindent{\bf Subcase 1.1} $u=v_1$ or $v_2$.

Without loss of generality, we assume $u=v_1$. Then $d_{v_1}\geq3$.
Let $\mathcal{H}'=(V_{\mathcal{H}},E_{\mathcal{H}'})$ be the hypergraph with $E_{\mathcal{H}'}=(E_{\mathcal{H}}\setminus \{e_1\})\cup \{e_1'\}$, where $e_1'=(e_1\setminus\{v_1\})\cup\{u_s\}$. By Proposition \ref{Prop:6}, $\mathcal{H}'$ is also a bicyclic $k$-uniform hypergraph. Note that $\pi(\mathcal{H})=(\ldots, d_{v_1}, \ldots, d_{u_s},\ldots)$,
$\pi(\mathcal{H}')=(\ldots, d_{v_1}-1, \ldots, d_{u_s}+1,\ldots)$, and $d_{v_1}\geq 3=d_{u_s}+2$. Then, by Lemma \ref{Lem:1}, we obtain that $h(\mathcal{H})>h(\mathcal{H}')$, a contradiction.

\noindent{\bf Subcase 1.2} $u\neq v_1,v_2$.

Let $f_1\not=e_1$ denote another edge containing $u$.
Let $\mathcal{H}'=(V_{\mathcal{H}},E_{\mathcal{H}'})$ be the hypergraph with $E_{\mathcal{H}'}=(E_{\mathcal{H}}\setminus \{e_1, f_1\})\cup \{e_1',f_1'\}$, where $e_1'=(e_1\setminus\{v_1\})\cup\{u_s\}$ and $f_1'=(f_1\setminus\{u\})\cup\{v_1\}$. By Proposition \ref{Prop:6}, $\mathcal{H}'$ is also a bicyclic $k$-uniform hypergraph. Note that $\pi(\mathcal{H})=(\ldots, d_{u}, \ldots, d_{u_s},\ldots)$, $\pi(\mathcal{H}')=(\ldots, d_u-1, \ldots, d_{u_s}+1,\ldots)$, and $d_{u}\geq 3=d_{u_s}+2$. Then, by Lemma \ref{Lem:1}, we obtain that $h(\mathcal{H})>h(\mathcal{H}')$, a contradiction.

\noindent{\bf Case 2.} $u\notin C$.

Find a longest path $Q=vf_1u_1f_2\cdots f_su_s$ ($u_i\notin C$ and $f_i\notin C$ for $i=1,2,\ldots, s$) starting at $v$ such that $v\in e_i$ for some $i\in \{1,2,\ldots, t\}$ and $f_{j-1}\cap f_j=\{u\}$ for some $j\in \{1,2,\ldots, s\}$.
Let $\mathcal{H}'=(V_{\mathcal{H}},E_{\mathcal{H}'})$ be the hypergraph with $E_{\mathcal{H}'}=(E_{\mathcal{H}}\setminus \{e_i, f_j\})\cup \{e_i',f_j'\}$, where $e_i'=(e_i\setminus\{v_i\})\cup\{u_s\}$ and $f_j'=(f_j\setminus\{u\})\cup\{v_i\}$. By Proposition \ref{Prop:6}, $\mathcal{H}'$ is also a bicyclic $k$-uniform hypergraph. Note that $\pi(\mathcal{H})=(\ldots,d_u,\ldots,d_{u_s},\ldots)$, $\pi(\mathcal{H}')=(\ldots,d_u-1,\ldots,d_{u_s}+1,\ldots)$, and $d_u\geq 3=d_{u_s}+2$. Then, by Lemma \ref{Lem:1}, we obtain that $h(\mathcal{H})>h(\mathcal{H}')$, a contradiction.

Thus, the discussions above imply that only $\mathcal{H}\in \mathcal{H}^{III}$ can attain the minimum value among all bicyclic $k$-uniform ($k\geq3$) hypergraphs on $n$ vertices with $m=\frac{n+1}{k-1}\ge2$ edges. By simple computation, we have $h(\mathcal{H})=2(m+1)\log2$ for any $\mathcal{H}\in \mathcal{H}^{III}$. This completes the proof of (a).\medskip

Next, we give the proof of (b). Suppose that $\mathcal{H}$ attains the maximum value among all bicyclic $k$-uniform ($k\geq3$) hypergraphs on $n$ vertices with $m=\frac{n+1}{k-1}\ge2$ edges. Let $u$ be a vertex of the maximum degree.\medskip

\noindent{\bf Claim 1.} \emph{If $e\in E_{\mathcal{H}}$ has $k-1$ pendent vertices, then $u\in e$.} \medskip

Otherwise, suppose that $e\in E_{\mathcal{H}}$ has $k-1$ pendent vertices but $u\notin e$. Let $w\in e$ be the vertex with $d_w>1$. Construct a new hypergraph $\mathcal{H}'=(V_{\mathcal{H}},E_{\mathcal{H}'})$ with $E_{\mathcal{H}'}=(E_{\mathcal{H}}\setminus \{e\})\cup \{e'\}$ where $e'=(e\setminus \{w\})\cup \{u\}$ and $d_w>1$. By Proposition \ref{Prop:6}, $\mathcal{H}'$ is also a bicyclic $k$-uniform hypergraph. Note that $\pi(\mathcal{H}) = (\ldots,d_{u},\ldots,d_{w},\ldots)$, $\pi(\mathcal{H}')=(\ldots,d_{u}+1,\ldots,d_{w}-1,\ldots)$, and $d_u>d_w-1$. By Lemma \ref{Lem:3}, we have $h(\mathcal{H})< h(\mathcal{H}')$, a contradiction.
\medskip

\noindent{\bf Claim 2.} \emph{$u$ is a common vertex of all the cycles of $\mathcal{H}$.} \medskip

Otherwise, assume $u$ does not lie on the cycle $C=v_1e_1v_2e_2\cdots e_tv_{t+1}~(v_{t+1}=v_1)$ of $\mathcal{H}$. Find a path $P_0=vf_1u_1f_2\cdots u_{s-1}f_su$ ($u_i\notin C$ for $i=1,2,\ldots, s-1$ and $f_j\notin C$ for $j=1,2,\ldots, s$) starting at $v$ such that $v\in e_i$ for some $i\in \{1,2,\ldots, t\}$.
Let $\mathcal{H}'=(V_{\mathcal{H}},E_{\mathcal{H}'})$ be the hypergraph with $E_{\mathcal{H}'}=(E_{\mathcal{H}}\setminus \{e_i,f_1\})\cup \{e_i',f_1'\}$ where $e_i'=(e_i\setminus \{v_i\})\cup \{u\}$, $f_1'=(f_1\setminus \{v\})\cup \{v_i\}$. By Proposition \ref{Prop:6}, $\mathcal{H}'$ is also a bicyclic $k$-uniform hypergraph. Note that $\pi(\mathcal{H}) = (\ldots,d_{u},\ldots,d_{v},\ldots)$,
$\pi(\mathcal{H}')=(\ldots,d_{u}+1,\ldots,d_{v}-1,\ldots)$, and $d_u>d_v-1$. By Lemma \ref{Lem:3}, we get that  $h(\mathcal{H})< h(\mathcal{H}')$, a contradiction.
\medskip

\noindent{\bf Claim 3.} \emph{$|e(C)|=2$ for any cycle $C$ in $\mathcal{H}$,  where $e(C)=\{e ~|~ e \text{ is an edge} \in C\}$. }\medskip

Again, we use $C=v_1e_1v_2e_2\cdots e_tv_{t+1}~(v_{t+1}=v_1)$ to denote a cycle in $\mathcal{H}$. Without loss of generality, suppose that  $v_2\not=u\in e_1$ ($u$ can be equal to $v_1$). If $|e(C)|\ge3$,
then let $\mathcal{H}'=(V_{\mathcal{H}},E_{\mathcal{H}'})$ be the hypergraph with $E_{\mathcal{H}'}=(E_{\mathcal{H}}\setminus \{e_2\})\cup \{e_2'\}$ where $e_2'=(e_2\setminus \{v_2\})\cup \{u\}$. By Proposition \ref{Prop:6}, $\mathcal{H}'$ is also a bicyclic $k$-uniform hypergraph.  Note that $\pi(\mathcal{H}) = (d_u,\ldots,d_{v_2},\ldots)$,
$\pi(\mathcal{H}')=(d_u+1,\ldots,d_{v_2}-1,\ldots)$, and $d_u>d_{v_2}-1$. By Lemma \ref{Lem:3}, we obtain that $h(\mathcal{H})< h(\mathcal{H}')$, a contradiction.

%If $|e(C)|\geq4$, then let $\mathcal{H}'=(V_{\mathcal{H}},E_{\mathcal{H}'})$ be the hypergraph with $E_{\mathcal{H}'}=(E_{\mathcal{H}}\setminus \{e_3\})\cup \{e_3'\}$ where $e_3'=(e_3\setminus \{v_4\})\cup \{u\}$. By Proposition \ref{Prop:6},  $\mathcal{H}'$ is also a $k$-uniform bicyclic hypergraph. Note that $\pi(\mathcal{H}) = (d_u,\ldots,d_{v_4},\ldots)$, $\pi(\mathcal{H}')=(d_u+1,\ldots,d_{v_4}-1,\ldots)$, and $d_u>d_{v_4}-1$.  By Lemma \ref{Lem:3}, wehave $h(\mathcal{H})< h(\mathcal{H}')$, a contradiction. \medskip

By Claims 1-3, we obtain that $h(\mathcal{H})$ attains the maximum value among all bicyclic $k$-uniform ($k\geq3$) hypergraphs on $n$ vertices with $m=\frac{n+1}{k-1}\ge2$ edges if and only if $\mathcal{H}\cong \mathcal{H}^{IV}$ or $\mathcal{H}\cong \mathcal{H}^{V}$. By simple computation, we have $h(\mathcal{H})= m\log m+4\log 2$ for $\mathcal{H}\cong \mathcal{H}^{IV}$ or $\mathcal{H}\cong \mathcal{H}^{V}$. This completes the proof of (b).\qed
\end{proof}

By Lemma \ref{Lem:6} and Equality (\ref{Equ:3.1}), we have the following result.

\begin{theorem}\label{The:3}
Let $\mathcal{H}$ be a bicyclic $k$-uniform ($k\geq3$) hypergraph on $n$ vertices with $m=\frac{n+1}{k-1}\ge2$ edges, and let $\mathcal{H}^{III}$, $\mathcal{H}^{IV}$, $\mathcal{H}^{V}$ be as in Lemma \ref{Lem:6}. Then %(with $m=n'$ edges, where $n'=\frac{n}{k-1}$).
$$
\log(km)-\frac{m\log m+4\log 2}{km} \le I_d^1(\mathcal{H})\leq \log(km)-\frac{2(m+1)\log2}{km},
$$
where the first equality holds if and only if $\mathcal{H}\cong  \mathcal{H}^{IV}$ or $\mathcal{H}\cong  \mathcal{H}^{V}$, and the second equality holds if and only if $\mathcal{H}\in  \mathcal{H}^{III}$.
\end{theorem}

\end{document}